\documentclass[12pt,reqno,draft]{amsart} 
\usepackage{amssymb,amscd,url}

\begin{document}

\allowdisplaybreaks

\newif\ifdraft 
\drafttrue
\newcommand{\DRAFTNUMBER}{3}
\newcommand{\DATE}{\today\ \ifdraft(Draft \DRAFTNUMBER)\fi}
\newcommand{\TITLE}{Primitive Divisors in Arithmetic Dynamics}
\newcommand{\TITLERUNNING}{Primitive Divisors in Arithmetic Dynamics}

\hyphenation{ca-non-i-cal archi-me-dean}


\newtheorem{theorem}{Theorem}
\newtheorem{lemma}[theorem]{Lemma}
\newtheorem{conjecture}[theorem]{Conjecture}
\newtheorem{question}[theorem]{Question}
\newtheorem{proposition}[theorem]{Proposition}
\newtheorem{corollary}[theorem]{Corollary}
\newtheorem*{claim}{Claim}

\theoremstyle{definition}
\newtheorem*{definition}{Definition}
\newtheorem{example}[theorem]{Example}

\theoremstyle{remark}
\newtheorem{remark}[theorem]{Remark}
\newtheorem*{acknowledgement}{Acknowledgements}



\newenvironment{notation}[0]{%
  \begin{list}%
    {}%
    {\setlength{\itemindent}{0pt}
     \setlength{\labelwidth}{2\parindent}
     \setlength{\labelsep}{\parindent}
     \setlength{\leftmargin}{3\parindent}
     \setlength{\itemsep}{0pt}
     }%
   }%
  {\end{list}}

\newenvironment{parts}[0]{%
  \begin{list}{}%
    {\setlength{\itemindent}{0pt}
     \setlength{\labelwidth}{1.5\parindent}
     \setlength{\labelsep}{.5\parindent}
     \setlength{\leftmargin}{2\parindent}
     \setlength{\itemsep}{0pt}
     }%
   }%
  {\end{list}}
\newcommand{\Part}[1]{\item[\upshape#1]}

\renewcommand{\a}{\alpha}
\renewcommand{\b}{\beta}
\newcommand{\g}{\gamma}
\renewcommand{\d}{\delta}
\newcommand{\e}{\epsilon}
\newcommand{\f}{\varphi}
\newcommand{\bfphi}{{\boldsymbol{\f}}}
\renewcommand{\l}{\lambda}
\renewcommand{\k}{\kappa}
\newcommand{\lhat}{\hat\lambda}
\newcommand{\m}{\mu}
\newcommand{\bfmu}{{\boldsymbol{\mu}}}
\renewcommand{\o}{\omega}
\renewcommand{\r}{\rho}
\newcommand{\rbar}{{\bar\rho}}
\newcommand{\s}{\sigma}
\newcommand{\sbar}{{\bar\sigma}}
\renewcommand{\t}{\tau}
\newcommand{\z}{\zeta}

\newcommand{\D}{\Delta}
\newcommand{\G}{\Gamma}
\newcommand{\F}{\Phi}

\newcommand{\ga}{{\mathfrak{a}}}
\newcommand{\gA}{{\mathfrak{A}}}
\newcommand{\gb}{{\mathfrak{b}}}
\newcommand{\gB}{{\mathfrak{B}}}
\newcommand{\gC}{{\mathfrak{C}}}
\newcommand{\gD}{{\mathfrak{D}}}
\newcommand{\gE}{{\mathfrak{E}}}
\newcommand{\gm}{{\mathfrak{m}}}
\newcommand{\gn}{{\mathfrak{n}}}
\newcommand{\go}{{\mathfrak{o}}}
\newcommand{\gO}{{\mathfrak{O}}}
\newcommand{\gp}{{\mathfrak{p}}}
\newcommand{\gP}{{\mathfrak{P}}}
\newcommand{\gq}{{\mathfrak{q}}}
\newcommand{\gR}{{\mathfrak{R}}}

\newcommand{\Abar}{{\bar A}}
\newcommand{\Ebar}{{\bar E}}
\newcommand{\Kbar}{{\bar K}}
\newcommand{\Pbar}{{\bar P}}
\newcommand{\Sbar}{{\bar S}}
\newcommand{\Tbar}{{\bar T}}
\newcommand{\ybar}{{\bar y}}
\newcommand{\phibar}{{\bar\f}}

\newcommand{\Acal}{{\mathcal A}}
\newcommand{\Bcal}{{\mathcal B}}
\newcommand{\Ccal}{{\mathcal C}}
\newcommand{\Dcal}{{\mathcal D}}
\newcommand{\Ecal}{{\mathcal E}}
\newcommand{\Fcal}{{\mathcal F}}
\newcommand{\Gcal}{{\mathcal G}}
\newcommand{\Hcal}{{\mathcal H}}
\newcommand{\Ical}{{\mathcal I}}
\newcommand{\Jcal}{{\mathcal J}}
\newcommand{\Kcal}{{\mathcal K}}
\newcommand{\Lcal}{{\mathcal L}}
\newcommand{\Mcal}{{\mathcal M}}
\newcommand{\Ncal}{{\mathcal N}}
\newcommand{\Ocal}{{\mathcal O}}
\newcommand{\Pcal}{{\mathcal P}}
\newcommand{\Qcal}{{\mathcal Q}}
\newcommand{\Rcal}{{\mathcal R}}
\newcommand{\Scal}{{\mathcal S}}
\newcommand{\Tcal}{{\mathcal T}}
\newcommand{\Ucal}{{\mathcal U}}
\newcommand{\Vcal}{{\mathcal V}}
\newcommand{\Wcal}{{\mathcal W}}
\newcommand{\Xcal}{{\mathcal X}}
\newcommand{\Ycal}{{\mathcal Y}}
\newcommand{\Zcal}{{\mathcal Z}}

\renewcommand{\AA}{\mathbb{A}}
\newcommand{\BB}{\mathbb{B}}
\newcommand{\CC}{\mathbb{C}}
\newcommand{\FF}{\mathbb{F}}
\newcommand{\GG}{\mathbb{G}}
\newcommand{\NN}{\mathbb{N}}
\newcommand{\PP}{\mathbb{P}}
\newcommand{\QQ}{\mathbb{Q}}
\newcommand{\RR}{\mathbb{R}}
\newcommand{\ZZ}{\mathbb{Z}}

\newcommand{\bfa}{{\mathbf a}}
\newcommand{\bfb}{{\mathbf b}}
\newcommand{\bfc}{{\mathbf c}}
\newcommand{\bfe}{{\mathbf e}}
\newcommand{\bff}{{\mathbf f}}
\newcommand{\bfg}{{\mathbf g}}
\newcommand{\bfp}{{\mathbf p}}
\newcommand{\bfr}{{\mathbf r}}
\newcommand{\bfs}{{\mathbf s}}
\newcommand{\bft}{{\mathbf t}}
\newcommand{\bfu}{{\mathbf u}}
\newcommand{\bfv}{{\mathbf v}}
\newcommand{\bfw}{{\mathbf w}}
\newcommand{\bfx}{{\mathbf x}}
\newcommand{\bfy}{{\mathbf y}}
\newcommand{\bfz}{{\mathbf z}}
\newcommand{\bfA}{{\mathbf A}}
\newcommand{\bfF}{{\mathbf F}}
\newcommand{\bfB}{{\mathbf B}}
\newcommand{\bfD}{{\mathbf D}}
\newcommand{\bfG}{{\mathbf G}}
\newcommand{\bfI}{{\mathbf I}}
\newcommand{\bfM}{{\mathbf M}}
\newcommand{\bfzero}{{\boldsymbol{0}}}

\newcommand{\Adele}{\textsf{\upshape A}}
\newcommand{\Aut}{\operatorname{Aut}}
\newcommand{\Br}{\operatorname{Br}}  
\newcommand{\Disc}{\operatorname{Disc}}
\newcommand{\density}{{\boldsymbol\delta}}
\newcommand{\densitysup}{\overline{\density}}
\newcommand{\densityinf}{\underline{\density}}
\newcommand{\Div}{\operatorname{Div}}
\newcommand{\End}{\operatorname{End}}
\newcommand{\Fbar}{{\bar{F}}}
\newcommand{\FOD}{\textup{FOM}}
\newcommand{\FOM}{\textup{FOD}}
\newcommand{\Gal}{\operatorname{Gal}}
\newcommand{\GL}{\operatorname{GL}}
\newcommand{\Index}{\operatorname{Index}}
\newcommand{\Image}{\operatorname{Image}}
\newcommand{\liftable}{{\textup{liftable}}}
\newcommand{\hhat}{{\hat h}}
\newcommand{\Ksep}{K^{\textup{sep}}}
\newcommand{\Ker}{{\operatorname{ker}}}
\newcommand{\Lsep}{L^{\textup{sep}}}
\newcommand{\Lift}{\operatorname{Lift}}
\newcommand{\pp}{\operatorname{pp}}  
\newcommand{\vlim}{\operatornamewithlimits{\text{$v$}-lim}}
\newcommand{\wlim}{\operatornamewithlimits{\text{$w$}-lim}}
\newcommand{\MOD}[1]{~(\textup{mod}~#1)}
\newcommand{\Norm}{{\operatorname{\mathsf{N}}}}
\newcommand{\notdivide}{\nmid}
\newcommand{\normalsubgroup}{\triangleleft}
\newcommand{\odd}{{\operatorname{odd}}}
\newcommand{\onto}{\twoheadrightarrow}
\newcommand{\Orbit}{\mathcal{O}}
\newcommand{\ord}{\operatorname{ord}}
\newcommand{\Per}{\operatorname{Per}}
\newcommand{\PrePer}{\operatorname{PrePer}}
\newcommand{\PGL}{\operatorname{PGL}}
\newcommand{\Pic}{\operatorname{Pic}}
\newcommand{\Prob}{\operatorname{Prob}}
\newcommand{\Qbar}{{\bar{\QQ}}}
\newcommand{\rank}{\operatorname{rank}}
\newcommand{\Resultant}{\operatorname{Res}}
\renewcommand{\setminus}{\smallsetminus}
\newcommand{\SL}{\operatorname{SL}}
\newcommand{\Span}{\operatorname{Span}}
\newcommand{\tors}{{\textup{tors}}}
\newcommand{\Trace}{\operatorname{Trace}}
\newcommand{\twistedtimes}{\mathbin{%
   \mbox{$\vrule height 6pt depth0pt width.5pt\hspace{-2.2pt}\times$}}}
\newcommand{\UHP}{{\mathfrak{h}}}    
\newcommand{\Wreath}{\operatorname{Wreath}}
\newcommand{\<}{\langle}
\renewcommand{\>}{\rangle}

\newcommand{\longhookrightarrow}{\lhook\joinrel\longrightarrow}
\newcommand{\longonto}{\relbar\joinrel\twoheadrightarrow}
\newcommand{\pmodintext}[1]{~\textup{(mod~$#1$)}}

\newcommand{\Spec}{\operatorname{Spec}}
\renewcommand{\div}{{\operatorname{div}}}

\newcounter{CaseCount}
\Alph{CaseCount}
\def\Case#1{\par\vspace{1\jot}\noindent
\stepcounter{CaseCount}
\framebox{Case \Alph{CaseCount}.\enspace#1}
\par\vspace{1\jot}\noindent\ignorespaces}

\title[\TITLERUNNING]{\TITLE}
\date{\DATE}

\author[Patrick Ingram]{Patrick Ingram}
\email{pingram@math.utoronto.ca}
\address{Department of Mathematics, University of Toronto,
Toronto, Ontario, M5S 2E4 Canada}
\author[Joseph H. Silverman]{Joseph H. Silverman}
\email{jhs@math.brown.edu}
\address{Mathematics Department, Box 1917
         Brown University, Providence, RI 02912 USA}
\thanks{The second author's research supported by 
NSA H98230-04-1-0064 and NSF DMS-0650017}

\begin{abstract}
Let~$\f(z)\in\QQ(z)$ be a rational function of degree~$d\ge2$
with~$\f(0)=0$ and such that~$\f$ does not vanish to order~$d$ at~$0$.
Let~$\a\in\QQ$ have infinite orbit under iteration of~$\f$ and
write~$\f^n(\a)=A_n/B_n$ as a fraction in lowest terms.  We prove that
for all but finitely many~$n\ge0$, the numerator~$A_n$ has a primitive
divisor, i.e., there is a prime~$p$ such that~$p\mid A_n$
and~$p\notdivide A_i$ for all $i<n$. More generally, we prove an
analogous result when~$\f$ is defined over a number field and~$0$ is a
periodic point for~$\f$.
\end{abstract}


\maketitle

\section*{Introduction}
Let~$\Acal=(A_n)_{n\ge1}$ be a sequence of integers. 
A prime~$p$ is called a \emph{primitive divisor of~$A_n$} if
\[
  p\mid A_n\qquad\text{and}\qquad p\notdivide A_i\quad\text{for all $1\le i<n$.}
\]
The \emph{Zsigmondy set of~$\Acal$} is the set
\[
  \Zcal(\Acal) = \bigl\{n\ge1 : 
        \text{$A_n$ does not have a primitive divisor}\bigr\}.
\]
A classical theorem of Bang~\cite{bang} (for~$b=1$)
and Zsigmondy~\cite{MR1546236} in general
says that if~$a,b\in\ZZ$ are integers with~$a>b>0$, then 
\[
  \Zcal\bigl((a^n-b^n)_{n\ge1}\bigr)~\text{is a finite set}.
\]
Indeed assuming that~$\gcd(a,b)=1$, they prove
that~$\Zcal\bigl((a^n-b^n)_{n\ge1}\bigr)$ contains no~$n>6$, which is
a strong uniform bound.  This useful result has been extended and
generalized in many ways, for example to more general linear
recursions, to number fields, to elliptic curves, and to Drinfeld
modules,
see~\cite{%
MR1863855,
MR1503541,
MR1502458,
MR2220263,
arxivNT0609120,
hsia:drinfeld,
IngramEDSoCC,
ingramsilverman,
MR0223330,
MR0344221,
MR961918
}.
\par
In this note we prove a Bang-Zsigmondy result for sequences
associated to iteration of rational functions. For ease of exposition,
we state here a special case of our main result for dynanmical systems
over~$\QQ$. See Theorem~\ref{thm:dynBZ} for the general statement.

\begin{theorem}
\label{thm:finiteZset}
Let~$\f(z)\in\QQ(z)$ be a rational function of degree~$d\ge2$ such
that~$\f(0)=0$, but~$\f$ does not vanish to order~$d$ at~$z=0$.
Let~$\a\in\QQ$ be a point with infinite orbit under iteration
of~$\f$. For each~$n\ge1$, let~$\f^n$ denote the~$n$'th iterate
of~$\f$ and write
\[
  \f^n(\a) = \frac{A_n}{B_n}\in\QQ
\]
as a fraction in lowest terms. Then the dynamical Zsigmondy
set $\Zcal\bigl((A_n)_{n\ge0}\bigr)$ is finite.
\end{theorem}

\begin{remark}
Rice~\cite{rice07} investigates primitive divisors in the case that
\text{$\f(z)\in\ZZ[z]$} is a monic polynomial and~$\a\in\ZZ$.  (See
also~\cite{flatters07} for some similar results.)  For example, Rice
proves that if~$\f(z)\ne z^d$, if~$0$ is preperiodic for~$\f$, and
if~$\a\in\ZZ$ has infinite~$\f$-orbit, then
$\Zcal\bigl((\f^n(\a))_{n\ge0}\bigr)$ is finite.  Our
Theorems~\ref{thm:finiteZset} and~\ref{thm:dynBZ} are generalizations
of~\cite{rice07} to arbitrary rational maps over number fields.
(However, we do assume that~$0$ is periodic, while Rice allows~$0$ to
be preperiodic.)
\end{remark}

A key tool in the proof of Theorem~\ref{thm:finiteZset} is a dynamical
analog~\cite{MR1240603} of Siegel's theorem~\cite[IX.3.1]{MR1329092}
for integral points on elliptic curves. Continuing with the notation
from Theorem~\ref{thm:finiteZset}, the dynamical canonical
height~\cite[\S3.4]{silverman:ads} of~$\a$ is the limit
\[
  \lim_{n\to\infty} \frac{\log\max\bigl\{|A_n|,|B_n|\bigr\}}{d^n}
  = \hhat_\f(\a) > 0.
\]
The positivity is a consequence of the fact that~$\a$ has
infinite orbit.  A deeper result, proven in~\cite{MR1240603} as a
consequence of Roth's theorem, implies that
\begin{equation}
  \label{eqn:Andntohhat}
  \lim_{n\to\infty} \frac{|A_n|}{d^n} = \hhat_\f(\a) > 0,
\end{equation}
and an estimate of this sort is needed to prove
Theorems~\ref{thm:finiteZset} and~\ref{thm:dynBZ}. 
\par
Of course, there are many situations in which it is easy
to prove that~\eqref{eqn:Andntohhat}
holds, for example if~$\f(z)\in\ZZ[z]$ and~$\a\in\ZZ$.
In such cases the exact determination of the Zsigmondy set often becomes
an elementary exercise, see Example~\ref{example:emptyZset}
and some of the examples in~\cite{flatters07,rice07}.

\begin{remark}
The first question that one asks about the Zsigmondy set of a
sequence is whether it is finite. Theorems~\ref{thm:finiteZset}
and~\ref{thm:dynBZ} give an affirmative answer for certain sequences
defined by iteration of rational maps on~$\PP^1$. Assuming that the
Zsigmondy sets under consideration are finite, it is also
natural to ask for explicit upper bounds for
\[
  \#\Zcal(\Acal)\qquad\text{and}\qquad \max\Zcal(\Acal),
\]
where one hopes that the bounds depend only minimally on the
sequence. For example, Zsigmondy's original theorem says that for
integers~$a>b>0$, we have~$\max\Zcal(a^n-b^n)\le6$. A recent deep
result of Bilu, Hanrot and Voutier~\cite{MR1863855} extends this to
the statement the~$\max\Zcal(\Lcal)\le30$ for any nontrivial Lucas or
Lehmer sequence~$\Lcal$.  In this paper we are content to prove the
finiteness of certain dynamical Zsigmondy sets. We leave the question
of explicit and/or uniform bounds as a problem for future study.
\end{remark}

\begin{remark}
Tom Tucker has pointed out to the authors that the results of this
paper should be valid also for iteration of non-split functions
over~$\CC(T)$, and more generally over one-dimensional function fields
of characteristic~$0$, since in this setting
Benedetto~\cite{arxiv0510444} (for polynomial maps) and
Baker~\cite{arxiv0601046} (for rational maps) have recently proven
that points with infinite orbit have strictly positive canonical
height.
\end{remark}

The material in this article is divided into two sections.  In
Section~\ref{section:dynZsigthm} we state and prove our main theorem
via a sequence of lemmas, some of which may be of independent
interest.  Section~\ref{section:speculations} discusses variants of
our main theorem and raises questions, makes conjectures, and
indicates directions for further research.

\begin{acknowledgement}
The authors would like to thank Rob Benedetto for sketching the
construction described in Remark~\ref{remark:localvsglobal}, and
Graham Everest, Igor Shparlinski and Tom Tucker for their helpful
comments on the initial draft of this paper.
\end{acknowledgement}

\section{A dynamical Zsigmondy theorem}
\label{section:dynZsigthm}

In this section we state and prove our main theorem concerning
primitive divisors in sequences defined by iteration of certain types
of rational functions. We start by recalling that primitive divisors
in number fields are most appropriately defined using ideals, rather
than elements.

\begin{definition}
Let~$K$ be a number field and let~$\Acal=(\gA_n)_{n\ge1}$ be a sequence
of nonzero integral ideals. 
A prime ideal~$\gp$ is called
a \emph{primitive divisor of~$\gA_n$} if
\[
  \gp\mid \gA_n\qquad\text{and}\qquad 
  \gp\notdivide \gA_i\quad\text{for all $1\le i<n$.}
\]
The \emph{Zsigmondy set of~$\Acal$} is the set
\[
  \Zcal(\Acal) = \bigl\{n\ge1 : 
        \text{$\gA_n$ does not have a primitive divisor}\bigr\}.
\]
\end{definition}

We also recall some basic definitions from dynamical systems.

\begin{definition}
Let~$\f(z)\in K(z)$ be a rational function of degree~$d\ge2$,
which we may view as a morphism~\text{$\f:\PP^1_K\to\PP^1_K$}.  
\par
A point~$\g\in\PP^1(\Kbar)$ is \emph{periodic for~$\f$}
if~$\f^n(\g)=\g$ for some~$n\ge1$.  The smallest such~$n$ is called
the~\emph{$\f$-period of~$\g$}. A point of~$\f$-period~$1$ is
called a \emph{fixed point}.
\par
Similarly, we say that~$\g$ is \emph{preperiodic}
if~$\f^{m+n}(\g)=\f^m(\g)$ for some~$n\ge1$
and~$m\ge0$. Equivalently,~$\g$ is preperiodic if its
\emph{$\f$-orbit} $\bigl\{\a,\f(\a),\f^2(\a),\dots\bigr\}$ is finite.
\par
A point that is not preperiodic, i.e., that has infinite $\f$-orbit,
is called a \emph{wandering point}.
\par
Let~$\g$ be a point of $\f$-period~$k$. We say that~$\f$ is
\emph{of polynomial type at~$\g$} if
\begin{equation}
  \label{fngzgdpsiz}
  \f^k(z) = \g + \frac{(z-\g)^d}{\psi(z)}
  \quad\text{for some~$\psi(z)\in K[z]$ with $\psi(\g)\ne0$.}
\end{equation}
\end{definition}

\begin{remark}
A more intrinsic algebro-geometric definition is that~$\f$ is of
polynomial type at~$\g$ if the map~$\f^k:\PP^1\to\PP^1$ is totally
ramified at~$\g$. Equivalently, if~$\f(z)$ has the
form~\eqref{fngzgdpsiz} and if we move~$\g$ to~$\infty$ by
conjugating~$\f$ by the linear fractional transformation
$f(z)=1/(z-\g)$, then the following calculation
shows that~$\f^k$ becomes a polynomial:
\begin{align*}
  (f\circ\f^k\circ f^{-1})(z)
   &= \frac{1}{\f^k(f^{-1}(z))-\g}
   = \frac{1}{(f^{-1}(z)-\g)^d/\psi(f^{-1}(z))}\\
   &= \frac{1}{z^{-d}/\psi(z^{-1}+\g)}
   = z^d\psi(z^{-1}+\g)
  \in K[z].
\end{align*}
\end{remark}

\begin{remark}
It is an easy exercise using the Riemann-Hurwitz genus formula to show
that if~$\f$ is of polynomial type at~$\g$, then the~$\f$-period
of~$\g$ is at most~$2$, cf.~\cite[Theorem~1.7]{silverman:ads}.  We
will not need to use this fact, but mention it because it provides an
easy way to check if there exist any points at which a given map~$\f$
is of polynomial type.
\end{remark}

\begin{theorem}
\label{thm:dynBZ}
Let~$K$ be a number field and let~$\f(z)\in K(z)$ be a rational
function of degree~$d\ge2$.  Let~$\g\in K$ be a periodic point
for~$\f$ such that~$\f$ is not of polynomial type at~$\g$.  Let~$\a\in
K$ be a wandering point, i.e., a point with infinite~$\f$-orbit, and
for each~$n\ge1$, write the ideal
\[
  \bigl(\f^n(\a)-\g\bigr) = \gA_n\gB_n^{-1}
\]
as a quotient of relatively prime integral
ideals. \textup(If~$\f^n(\a)=\infty$, then set \text{$\gA_n=(1)$}
and~$\gB_n=(0)$.\textup) Then the dynamical Zsigmondy set
$\Zcal\bigl((\gA_n)_{n\ge1}\bigr)$ is finite.
\end{theorem}

\begin{remark}
The assumption in Theorem~\ref{thm:dynBZ} that~$\f$ is not of
polynomial type at~$\g$ is a necessary condition.  For example,
let~$F(z)\in \ZZ[z]$ be any polynomial of degree at most~$d$ with~$F(0)=1$
and consider the rational map
\[
  \f(z)=z^d/F(z).
\]
Then~$\f$ is of polyomial type at~$\g=0$, and an easy calculation shows
that for any starting value~$\a=A_1/B_1\in\QQ$, we have~$A_n=A_1^{d^n}$
for all~$n\ge0$. Hence~$A_n$ has no primitive divisors for
any~$n\ge1$.
\end{remark}

\begin{proof}
The proof of  Theorem~\ref{thm:dynBZ}
is structured as a series of lemmas that provide the necessary tools.

\begin{lemma}
If Theorem~$\ref{thm:dynBZ}$ is true when~$\g$ is a fixed point
of~$\f$, then it is true when~$\g$ is a periodic point of~$\f$.
\end{lemma}
\begin{proof}
Suppose that~$\g$ has~$\f$-period~$k\ge2$, so~$\f^k(\g)=\g$ and no smaller
iterate of~$\f$ fixes~$\g$. For each~$0\le i<k$ we consider the subsequence
\[
  \bigl(\f^{nk+i}(\a)-\g\bigr) = \gA_{nk+i}^{\vphantom1}\gB_{nk+i}^{-1}
  \quad\text{for~$n=0,1,2,\dots\,$.}
\]
We claim that these subsequences have very few common prime divisors.
More precisely, define the set of good primes~$\Pcal=\Pcal_{\f,\a,\g}$
to be primes satisfying the following two conditions:
\begin{enumerate}
\item[(A)]
$\f$ has good reduction at~$\gp$. (See~\cite[Chapter~2]{silverman:ads}
for the definition and basic properties of maps with good reduction.)
\item[(B)]
$\f^i(\g)\not\equiv\g\pmodintext{\gp}$ for all $0\le i<k$.
\end{enumerate}
It is clear that~$\Pcal$ contains all but finitely many primes.
Now suppose that some~$\gp\in\Pcal$ divides terms in different
subsequences, say
\begin{equation}
  \label{eqn:gpgAnki}
  \gp\mid \gA_{nk+i}
  \quad\text{and}\quad
  \gp\mid \gA_{mk+j}
  \quad\text{for  some $0\le j<i<k$.}
\end{equation}
Note that the good reduction assumption means
that~$(\f\bmod\gp)^n=(\f^n)\bmod\gp$, i.e., reduction modulo~$\gp$
commutes with composition of~$\f$
(see~\cite[Theorem~2.18]{silverman:ads}). 
Hence if~$\gp$ is a prime of good reduction for~$\f$, then
we have
\[
  \gp\mid\gA_n\quad\Longleftrightarrow\quad
  \f^n(\a)\equiv\g\pmodintext\gp.
\]
So we can rewrite assumption~\eqref{eqn:gpgAnki} as
\[
  \f^{nk+i}(\a) \equiv \f^{mk+j}(\a) \equiv \g \pmod\gp.
\]
Suppose first that~$nk+i>mk+j$.  Since~$0\le j<i\le k$ by assumption,
this implies that~$n\ge m$.  We compute
\begin{align*}
  \g &\equiv \f^{nk+i}(\a) \pmodintext\gp \\
  &= \f^{(nk+i)-(mk+j)}\bigl(\f^{mk+j}(\a)\bigr) \\
  &\equiv \f^{(nk+i)-(mk+j)}(\g) \pmodintext\gp \\
  &= \f^{i-j}\bigl( (\f^k)^{n-m}(\g) \bigr) \\
  &= \f^{i-j}(\g) \qquad\text{since $\f^k(\g)=\g$.}
\end{align*}
But~$0<i-j<k$, so this contradicts Property~(B) of~$\Pcal$.
\par
Similarly, if~$nk+i < mk+j$, then~$m>n$ (since~$i>j$), so we have
\begin{align*}
  \g &\equiv \f^{mk+j}(\a) \pmodintext\gp \\
  &= \f^{(mk+j)-(nk+i)}\bigl(\f^{nk+i}(\a)\bigr) \\
  &\equiv \f^{(mk+j)-(nk+i)}(\g) \pmodintext\gp \\
  &= \f^{j-i+k}\bigl( (\f^k)^{m-n-1}(\g) \bigr) \\
  &= \f^{j-i+k}(\g) \qquad\text{since $\f^k(\g)=\g$.}
\end{align*}
This is again a contradiction of Property~(B), since~$0<j-i+k<k$.
\par
We have now proven that for primes~$\gp\in\Pcal$, at most one of
the subsequences
\[
  (\gA_{nk+i})_{n\ge0},\qquad i=0,1,\ldots,k-1,
\]
has a term divisible by~$\gp$. 
\par
We are assuming that Theorem~\ref{thm:dynBZ} is true if~$\g$ is a
fixed point. It follows that for each~$0\le i<k$, the Zsigmondy
set~$\Zcal\bigl((\gA_{nk+i})_{n\ge0}\bigr)$ is finite,
since~$(\gA_{nk+i})_{n\ge0}$ is the sequence associated to the map~$\f^k$, the
initial point~$\f^i(\a)$, and the fixed point~$\g$ of~$\f^k$.  
(Note that the condition on~$\f$ is not a polynomial at~$\g$ is equivalent
to the condition that~$\f^k$ is not of polynomial type at~$\g$.)
\par
It follows that we can find a number~$N$ so that
for all~$n\ge N$ and all~$0\le i<k$ there is a prime ideal~$\gp_{n,i}$
satisfying
\[
  \gp_{n,i} \mid \gA_{nk+i}
  \quad\text{and}\quad
  \gp_{n,i} \notdivide \gA_{mk+i}
  \quad\text{for all $0\le m<n$.}
\]
In other words, for a fixed~$i$, the ideals~$\gp_{n,i}$ are primitive
divisors in the subsequence~$(\gA_{nk+i})_{n\ge0}$. Increasing~$N$ if
necessary, we may assume that~$\gp_{n,i}\in\Pcal$ for all~$n$ and
all~$i$, since the complement of~$\Pcal$ is finite.
\par
It is now clear that for~$n>N$ and~$0\le i<k$, the prime~$\gp_{n,i}$
is a primitive divisor of~$\gA_{nk+i}$ in the full
sequence~$(\gA_m)_{m\ge0}$. This is true because it is a primitive
divisor in its own subsequence, and we proved above that it does not
divide any of the terms in any of the other subsequences.
\end{proof}

We next reduce to the case~$\g=0$, which will simplify our later
computations.

\begin{lemma}
It suffices to prove Theorem~$\ref{thm:dynBZ}$
under the assumption that \text{$\g=0$}.
\end{lemma}
\begin{proof}
Let~$f(z)=z+\g$. Then we have
\[
  \f^n(\a) - \g = f^{-1}\bigl(\f^n(\a)\bigr)
  = (f^{-1}\circ\f\circ f)^n\bigl(f^{-1}(\a)\bigr).
\]
Hence replacing~$\f$ by~$f^{-1}\circ\f\circ f$ and replacing~$\a$
by~$f^{-1}(\a)$ allows us to replace~$\g$ wtih~$f^{-1}(\g)=0$.
Finally, we note that~$\f$ is of polynomial type at~$\g$ if and
only if~$f^{-1}\circ\f\circ f$ is of polynomial type at~$f^{-1}(\g)=0$,
so the conjugated map has the required property.
\end{proof}

We are now reduced to the case that~$\g=0$ is a fixed point
of~$\f(z)$. This means that we can write~$\f$ in the form
\begin{equation} 
  \label{eqn:fza1zb0}
  \f(z) = \frac{a_ez^e+a_{e+1}z^{e+1}+\cdots+a_dz^d}
        {b_0+b_1z+b_2z^2+\cdots+b_dz^d},
\end{equation}
with~$a_e\ne0$, where without loss of generality we may assume that
all~$a_i$ and~$b_i$ are in the ring of integers~$R$ of~$K$.  Further,
since~$\deg\f=d$, we know that at least one of~$a_d$ and~$b_d$ is
nonzero, and also~$b_0\ne0$.  Finally, our assumptions that~$\f(0)=0$
and that~$\f$ is not of polynomial type at~$0$ imply that
\begin{equation}
  \label{0lteltd}
  0 < e < d.
\end{equation}
Note the strict inequalities on both sides.
\par
We next prove an elementary, but useful, lemma that bounds how
rapidly the~$\gp$-divisibilty of~$\gA_n$ can grow, where recall
that~$\gA_n$ is the integral ideal obtained by writing
\[
  \bigl(\f^n(\a)\bigr) = \gA_n\gB_n^{-1}
\]
as a quotient of relatively prime integral ideals. In order to
state the result, we need one definition.

\begin{definition}
For each prime ideal~$\gp$, we define the \emph{rank of apparition}
(\emph{of~$\f$ and~$\a$}) \emph{at~$\gp$} to be the integer
\[
  r_\gp = \min\{r\ge0 : \ord_\gp\gA_r > 0\}.
\]
(If no such~$r$ exists, we set~$r_\gp=\infty$.)  Notice that this is a
direct analogy of Ward's definition~\cite{MR0023275} of the rank of
apparition for with elliptic divisibility sequences.
\end{definition}

\begin{lemma}
\label{lemma:pdivgrowth}
With notation as above, let
\[
  S = \bigl\{\gp : \ord_\gp(a_eb_0)\ne0\bigr\}.
\]
Then for all primes~$\gp\notin S$, 
\begin{align}
  k &\le r_\gp &\Longrightarrow&& \ord_\gp\gA_{k-1} &= 0,
         \label{eqn:ordgAsmall}  \\  
  k &> r_\gp   &\Longrightarrow&& \ord_\gp\gA_{k}&=e\ord_\gp\gA_{k-1},
         \label{eqn:ordgAbig}  
\end{align}
where~$e$ is the order of vanishing of~$\f(z)$ at~$z=0$,
see~\eqref{eqn:fza1zb0}.
\end{lemma}
\begin{proof}
We note that~\eqref{eqn:ordgAsmall} is true by the definition
of~$r_\gp$, so we only need to prove~\eqref{eqn:ordgAbig},
which we do by induction on~$k$. Since~$\gA_{r_\gp}>0$ by definition,
the inductive hypothesis implies that
\begin{equation}
  \label{eqn:gAn1mpma1pb0}
   \ord_\gp\gA_i>0\qquad\text{for all $r_\gp\le i<k$.}
\end{equation}
In particular, \text{$\ord_\gp\gA_{k-1}>0$}.
\par
For notational convenience, we let~$\b=\f^{k-1}(\a)$.  The fact
that~$\gA_{k-1}$ has positive valuation implies that
\[
  \ord_\gp\b  =\ord_\gp\gA_{k-1} > 0.
\]
Further, the assumption that~$\gp\notin S$ means that
\[
  \ord_\gp a_e = \ord_\gp b_0 = 0,
\]
so when we evaluate the numerator and denominator
of~$\f(\b)$, the lowest degree terms have the strictly
smallest~$\gp$-adic valuation. Hence the ultrametric triangle
inequality gives
\begin{align*}
  \ord_\gp(\gA_{k})
  &=\ord_\gp\bigl(\f(\b)\bigr)\\
  &= \ord_\gp(a_e\b^e+a_{e+1}\b^{e+1}+\cdots+a_d\b^d) \\*
  &\qquad{}        - \ord_\gp(b_0+b_1\b+\cdots+b_d\b^d) \\
  &=e\ord_\gp(\b)
    \qquad\text{since $\ord_\gp(a_eb_0)=0$ and $\ord_\gp(\b)>0$.}\\*
  &=e\ord_\gp(\gA_{k-1}).
\end{align*}
This completes the proof of Lemma~\ref{lemma:pdivgrowth}.
\end{proof}

Next we recall the definition and basic properties of
the canonical height associated to~$\f$.

\begin{lemma}
\label{lemma:canoicalht}
The \emph{canonical height associated to~$\f$} is
the function \text{$\hhat_\f:\PP^1(\Kbar)\to\PP^1(\Kbar)$} defined by
the limit
\[
  \hhat_\f(\b) = \lim_{n\to\infty} \frac{1}{d^n}h\bigl(\f^n(\b)\bigr).
\]
It satisfies, and is characterized by, the two following properties\textup:
\begin{align}
  \hhat_\f(\b) &= h(\b)+O(1)&&\text{for all $\b\in\PP^1(\Kbar)$.}
  \label{eqn:ht1}\\
  \hhat_\f\bigl(\f(\b)\bigr) &= d\hhat_\f(\b)
          &&\text{for all $\b\in\PP^1(\Kbar)$.}
  \label{eqn:ht2}
\end{align}
The~$O(1)$ constant in~\eqref{eqn:ht1} depends on~$\f$, but is
independent of~$\b$.
\par
The values~$\hhat_\f(\b)$ are nonnegative, and
\[
  \hhat_\f(\b)>0 \quad\Longleftrightarrow\quad
  \text{$\b$ has infinite $\f$-orbit}.
\]
\end{lemma}
\begin{proof}
See~\cite{callsilv:htonvariety}, \cite[\S B.4]{MR1745599},
or~\cite[\S3.4]{silverman:ads}.
\end{proof}

\begin{definition}
Let~$S$ be a finite set of places of~$K$, including all archimedean places
and let~$\gA$ be an integral ideal. The \emph{prime-to-$S$ norm of~$\gA$} is
the quantity
\[
  \Norm_S\gA = \prod_{\gp\notin S} \gp^{\ord_\gp\gA}.
\]
As the name suggests,~$\Norm_S\gA$ is the part of~$\Norm_{K/\QQ}\gA$ that
is relatively prime to all of the primes in~$S$.
\end{definition}

We next apply the main result of~\cite{MR1240603} to show that
$\log(\Norm_S\gA_n)$ grows like a constant multiple of~$d^n$. We observe
that the proof of~\cite[Theorem~E]{MR1240603} requires some sort of
nontrivial theorem on Diophantine approximation such as Roth's
theorem, so despite its simple statement, the following lemma conceals
the deepest part of the proof of Theorem~\ref{thm:dynBZ}. For our
purposes, we require the general number field version proven
in~\cite{MR1240603}, but see~\cite[\S3.8]{silverman:ads} for a more
leisurely exposition of the same result over~$\QQ$
with~$S=\{\infty\}$.

\begin{lemma}
\label{lemma:htvslognorm}
\textup{(a)}
There is a constant~$C=C(\f,\a)$ so that
\begin{equation}
  \label{eqn:1ednlnAn1}
  \frac{1}{[K:\QQ]} \log\Norm_{K/\QQ}\gA_n
  \le d^n\hhat_\f(\a) + C
  \qquad\text{for all $n\ge0$.}
\end{equation}
\par\noindent\textup{(b)}
Let~$S$ be a finite set of places, including all archimedean places,
and let~$\e>0$. There is an~$n_0=n_0(\e,S,\f,\a)$ so that 
\begin{equation}
  \label{eqn:1ednlnAn2}
  \frac{1}{[K:\QQ]} \log\Norm_S\gA_n
  \ge (1-\e)d^n\hhat_\f(\a)
  \qquad\text{for all $n\ge n_0$.}
\end{equation}
\end{lemma}

\begin{remark}
We observe that the elementary upper bound~\eqref{eqn:1ednlnAn1} is
true for any rational map, while the deeper lower
bound~\eqref{eqn:1ednlnAn2} requires the assumption that~$\f$ is not
of polynomial type at~$0$.
\end{remark}

\begin{proof}
(a) 
In general, if~$\b\in K^*$ and if we write the ideal~$(\b)$
as a quotient of relatively prime integral ideals~$(\b)=\gA\gB^{-1}$,
then the (normalized logarithmic) height of~$\b$ is given by
the formula
\begin{equation}
  \label{eqn:hbnorm}
  h(\b) = \frac{1}{[K:\QQ]} \Bigl(\log\Norm_{K/\QQ}\gB
             + \sum_{v\in M_K^\infty} [K_v:\QQ_v]\log\max\{1,|\b|_v\} \Bigr).
\end{equation}
See for example~\cite[\S3.1]{LangDG}.  We apply this
with~$\b=\f^n(\a)$, so~$\gB=\gA_n$, and we use the fact that all of
the terms in the sum are non-negative to deduce that
\begin{equation}
  \label{eqn:hfn1}
  h\bigl(\f^n(\a)\bigr) \ge \frac{1}{[K:\QQ]}\log\Norm_{K/\QQ}\gA_n.
\end{equation}
Finally, we use Lemma~\ref{lemma:canoicalht} to compute
\begin{equation}
  \label{eqn:hfn2}
  h\bigl(\f^n(\a)\bigr) = \hhat_\f\bigl(\f^n(\a)\bigr) + O(1)
  = d^n\hhat_\f(\a)+O(1).
\end{equation}
Combining~\eqref{eqn:hfn1} and~\eqref{eqn:hfn2} completes
the proof of~(a).
\par\noindent(b)
Our assumption that~$\f$ is not of polynomial type at~$0$ allows us to
apply~\cite[Theorem~E]{MR1240603}. This theorem implies that
for each~$v\in S$ we have
\begin{equation}
  \label{eqn:limdfna}
  \lim_{n\to\infty} \frac{\d_v\bigl(\f^n(\a),0\bigr)}{d^n}
  = 0,
\end{equation}
where~$\d_v$ is a logarithmic $v$-adic distance function
on~$\PP^1(K_v)$. Since we are measuring the distance to~$0$, we may
take~$\d_v$ to be the function
\[
  \d_v(\b,0) = 1 + \log\left(\frac{\max\{|\b|_v,1\}}{|\b|_v}\right)
           = 1 + \log \max\{1,|\b|_v^{-1}\}.
\]
(See~\cite[\S3]{MR1240603}.) Substituting this into~\eqref{eqn:limdfna},
we obtain the equivalent statement
\begin{equation}
  \label{eqn:limdfna1}
  \lim_{n\to\infty} \frac{\log \max\{1,|\f^n(\a)|_v^{-1}\}}{d^n}
  = 0.
\end{equation}
\par
We now rewrite~\eqref{eqn:hbnorm}, moving the part of the norm coming
from primes in~$S$ into the sum. Using our notation for the
prime-to-$S$ norm,  formula~\eqref{eqn:hbnorm} becomes
\begin{equation}
  \label{eqn:hbnorm1}
  h(\b) = \frac{1}{[K:\QQ]} \Bigl(\log\Norm_S\gB
             + \sum_{v\in S} [K_v:\QQ_v]\log\max\{1,|\b|_v\} \Bigr).
\end{equation}
We apply~\eqref{eqn:hbnorm1} with~$\b=\f^n(\a)^{-1}$, so~$\gB=\gA_n$.
Since~$h(\b)=h(\b^{-1})$ for any nonzero~$\b$, this gives
\begin{multline}
  \label{eqn:hanorm}
  h\bigl(\f^n(\a)\bigr) = \frac{1}{[K:\QQ]} \log\Norm_S\gA_n \\*
     + \sum_{v\in S} \frac{[K_v:\QQ_v]}{[K:\QQ]}
       \log\max\{1,|\f^n(\a)|^{-1}_v\}.
\end{multline}
We now divide both sides of~\eqref{eqn:hanorm} by~$d^n$ and let~$n\to\infty$.
The limit formula~\eqref{eqn:limdfna1} tells us that the sum
over the places in~$S$ goes to~$0$. On the other hand, the
left-hand side of~\eqref{eqn:hanorm} is exactly the limit that defines the
canonical height. Hence we obtain
\[
  \hhat_\f(\a) 
  = \lim_{n\to\infty} \frac{1}{[K:\QQ]} \frac{\log\Norm_{K/\QQ}\gA_n}{d^n}.
\]
This limit implies the lower bound~\eqref{eqn:1ednlnAn2} that we are
trying to prove, which completes the proof of
Lemma~\ref{lemma:htvslognorm}. (It also implies an upper bound, but a
weaker upper bound than we obtained by the elementary argument in~(a).)
\end{proof}

We have assembled all of the tools needed to complete the proof of
Theorem~\ref{thm:dynBZ}. Our goal is to show that~$\gA_n$ has a
primitive prime divisor for all sufficiently large~$n$.  
In order to do this, we define an ideal
\begin{equation}
  \label{eqn:prod1}
  \gE_{n,S} := 
  \prod_{\substack{\text{primes $\gp\notin S$ that divide}\\ 
      \text{one of $\gA_0,\gA_1,\dots,\gA_{n-1}$}\\}}
                   \gp^{\ord_\gp\gA_n}.
\end{equation}
Similarly, we let~$\gA_{n,S}$ be the prime-to-$S$ part of~$\gA_n$, thus
\[
  \gA_{n,S} = \prod_{\gp\notin S} \gp^{\ord_\gp\gA_n}.
\]
We will prove that for all sufficiently large~$n$, the
ideal~$\gA_{n,S}$ is strictly larger than the ideal~$\gE_{n,S}$.  This
will imply the desired result, since it will show that~$\gA_n$ has a
primitive prime divisor, and indeed a primitive prime divisor not
lying in~$S$.
\par
Suppose that~$\gp\notin S$ divides~$\gA_i$ for some~$i<n$.
Then~$r_\gp\le i<n$, so Lemma~\ref{lemma:pdivgrowth} 
tells us that
\[
  \ord_\gp\gA_{n-1}>0 \qquad\text{and}\qquad
  \ord_\gp\gA_n = e\ord_\gp\gA_{n-1}.
\]
Thus in the product~\eqref{eqn:prod1} defining~$\gE_{n,S}$,
it suffices to multiply over the primes dividing~$\gA_{n-1}$, and we
find that
\[
  \gE_{n,S} = \prod_{\gp\notin S,\;\gp\mid\gA_{n-1}} \gp^{\ord_\gp\gA_n}
  = \prod_{\gp\notin S,\;\gp\mid\gA_{n-1}} \gp^{e\ord_\gp\gA_{n-1}}
  = \gA_{n-1,S}^e.
\]
The upper bound in Lemma~\ref{lemma:htvslognorm}(a), applied
to~$\gA_{n-1}$, gives the estimate
\begin{align}
  \label{eqn:lognorm1}
  \log\Norm_{K/\QQ}\gE_{n,S}
  & = e\log\log\Norm_{K/\QQ}\gA_{n-1,S}\notag\\
  &\le e\log \Norm_{K/\QQ}\gA_{n-1} \notag\\
  &\le [K:\QQ]ed^{n-1}\hhat_\f(\a) - O(1),
\end{align}
where the~$O(1)$ depends only on~$\f$,~$\a$, and~$K$.
\par
To obtain the complementary lower bound, we apply
Lemma~\ref{lemma:htvslognorm}(b) to~$\gA_n$ with~$\e=\frac{1}{2d}$. This
gives
\begin{equation}
  \label{eqn:lognorm2}
  \log\Norm_{K/\QQ}\gA_{n,S}
  = \log\Norm_S\gA_n
  \ge [K:\QQ](1-\tfrac{1}{2d})d^n\hhat_\f(\a),
\end{equation}
valid for~$n\ge n_0(S,\f,\a)$.
\par
Finally, we combine~\eqref{eqn:lognorm1} and~\eqref{eqn:lognorm2} 
to obtain
\begin{align*}
  \log\frac{\Norm_{K/\QQ}\gA_{n,S}}{\Norm_{K/\QQ}\gE_{n,S}}
  &\ge \left(1-\frac{1}{2d}\right)d^n\hhat_\f(\a)
         - [K:\QQ]ed^{n-1}\hhat_\f(\a) + O(1) \\
  &\ge \frac12[K:\QQ]d^{n-1}\hhat_\f(\a) + O(1),
  \quad\text{since $e<d$ from~\eqref{0lteltd}.}
\end{align*}
The right-hand side is positive for all sufficiently large~$n$,
which completes the proof that~$\gA_n$ has a primitive divisor
for all sufficiently large~$n$, and
hence that the Zsigmondy set~$\Zcal\bigl((\gA_n)_{n\ge1}\bigr)$ is
finite.
\end{proof}

\begin{example}
\label{example:emptyZset}
It is easy to construct specific examples, and even families of examples,
for which one can find the full dynamical Zsgimondy set by an
elementary calculation. We illustrate with one such example,
see~\cite{flatters07,rice07} for others.
\par
Let~$\f(z)=z^2+z$, let~$\g=0$, and let~$\a\in\QQ$ with~$\a>0$.
Then for all~$n\ge1$ we have~$\gcd(A_{n-1},B_{n-1})=1$ and
\[
  \frac{A_{n}}{B_{n}} = \frac{A_{n-1}^2+A_{n-1}B_{n-1}}{B_{n-1}^2},
\]
so~$A_{n}=A_{n-1}^2+A_{n-1}B_{n-1}$ and~$B_{n}=B_{n-1}^2$. In particular, 
\[
  A_0\mid A_1\mid A_2\mid\cdots
\]
so~$A_n$ has a primitive prime divisor if and only if there is a
prime~$p\mid A_n$ with~$p\notdivide A_{n-1}$.  But
from~$A_{n}=A_{n-1}(A_{n-1}+B_{n-1})$, this means that~$A_n$ has a
primitive prime divisor if and only if~$A_{n-1}+B_{n-1}>1$. This is
true for every~$n\ge1$, so it follows
that~$\Zcal\bigl((A_n)_{n\ge1}\bigr)=\emptyset$.
\end{example}

\section{Questions and Speculations}
\label{section:speculations}
It is interesting to ask to what extent one can relax or modify the
hypotheses in Theorem~\ref{thm:dynBZ}. We discuss this question in a
series of remarks.

\begin{remark}
Presumably Theorem~\ref{thm:dynBZ} remains true if we only assume
that~$\g$ is preperiodic, i.e., has a finite $\f$-orbit, rather than
assuming that~$\g$ is periodic. However, some care is needed in order
to generalize the argument, so we have not pursued it here. We observe
that Rice's main result~\cite[Theorem~1.1]{rice07}
for monic~$\f(z)\in\ZZ[z]$ and~$\a\in\ZZ$ permits~$\g$ to be preperiodic,
albeit in a situation where it is relatively easy to classify all
situations for which there exist nonperiodic preperiodic points.
\end{remark}

\begin{remark}
\label{remark:localvsglobal}
If~$\gA_n$ is a classical divisibility sequence associated to a
(possibly twisted) multiplicative group or to an elliptic curve, then
the higher order~$\gp$ divisibility can be can be described quite
precisely via properties of the formal group of the group
variety. Roughly speaking, if~$\gA_r$ is the first term divisible
by~$\gp$, then
\begin{equation}
  \label{eqn:classicalordp}
  \ord_\gp(\gA_{kr}) = \ord_\gp(\gA_r) + \ord_\gp(k)
  \quad\text{for all $k\ge1$,}
\end{equation}
and no other terms are divisible by~$\gp$. The proof
of~\eqref{eqn:classicalordp} is essentially $\gp$-adic in nature, and
indeed it holds for classical divisibility sequences defined over the
$\gp$-adic completion~$K_\gp$ of~$K$. We note that the
estimate~\eqref{eqn:classicalordp} forms an essential, albeit
reasonably elementary, component of the proof that classical Zsigmondy
sets are finite.
\par
In the proof of our main theorem, Lemma~\ref{lemma:pdivgrowth}
plays the role of~\eqref{eqn:classicalordp}, but note
that Lemma~\ref{lemma:pdivgrowth}, which implies that
\[
  \ord_\gp\gA_k = e^{k-r}\ord_\gp\gA_r,
\]
is only valid for primes outside a certain bad set~$S$.  It is natural
to ask if the dynamical analog of~\eqref{eqn:classicalordp} is true
over~$K_\gp$, since that might lead to a better understanding of the
underlying dynamics. The answer is probably not.  Rob Benedetto
(private communication) has sketched an argument using ideas
from~\cite{MR1941304,MR2285248} which suggests that there are
dynamical examples over~$\QQ_p$ such that~$\ord_p\gA_n$ grows grows
extremely rapidly, for exampe faster than~$O(d^n)$, or even faster
than~$O(2^{d^n})$.
\par
This is in marked contrast to the situation over a number field,
where the elementary height estimate
\begin{multline*}
  \ord_\gp\f^n(\a)\log \Norm\gp
  \le\log \Norm\gA_n
  \le [K:\QQ]h\bigl(\f^n(\a)\bigr) \\
  = [K:\QQ]\hhat_\f\bigl(\f^n(\a)\bigr)+O(1) 
   = d^n[K:\QQ]\hhat_\f(\a)+O(1)
\end{multline*}
shows that $\ord_\gp\f^n(\a)$ cannot grow faster than~$O(d^n)$.
\end{remark}

\begin{remark}
\label{remark:gwandering}
If we change the assumptions of Theorem~\ref{thm:dynBZ} to make~$\g$
a wandering point, then Zsigmondy-type theorems appear to be
more difficult to prove. Note that when stripped to their essentials, 
proofs of Zsigmondy-type theorems have two main ingredients:
\begin{enumerate}
\item
Prove that the sequence~$\gA_n$ grows very rapidly.
\item
Prove that once a prime~$\gp$ divides some term in the sequence, it
cannot divide later terms to an extremely high power.
\end{enumerate}
Condition~1 is a global condition, and it is true for both preperiodic
and wandering~$\g$. This may be proved
using~\cite[Theorem~E]{MR1240603} as we did in the proof of
Lemma~\ref{lemma:htvslognorm}.  Indeed, the proof of
Lemma~\ref{lemma:htvslognorm} only uses the assumption that~$\f$
is not of polynomial type at~$\g$, it does not require that~$\g$
be a periodic point.
\par
Thus the difficulty in proving a Zsigmondy-type theorem for
wandering~$\g$ is Condition~(2). As noted in
Remark~\ref{remark:localvsglobal}, classical multiplicative or elliptic
divisibility sequences satisfy
\[
  \ord_\gp(\gA_{kr}) = \ord_\gp(\gA_r) + \ord_\gp(k),
\]
where~$r$ is the rank of apparition at~$\gp$.  However, in the
dynamical setting there is no analogous uniform result, and indeed it
is easy to construct examples in which $\ord_\gp(\gA_{kr})$ is
arbitrarily much larger than~$\ord_\gp(\gA_r)$. For example, let
\[
  \g=0,\quad
  \a=p,\quad\text{and}\quad
  \f(z)=z^2-pz+p^e.
\]
Then
\[
  \ord_p(\gA_0)=\ord_p(\a)=1
  \quad\text{and}\quad
  \ord_p(\gA_1)=\ord_p(p^e)=e.
\]
\end{remark}

\begin{remark}
\label{remark:a=g}
Our main theorem is really about primes~$\gp$ such that the reduction
modulo~$\gp$ of a wandering point~$\a$ coincides with a given periodic
point~$\g$. This is a natural question, but it is possibly not the
most natural generalization of the classical multipicative and
elliptic divisibility sequences. In the classical case, one studies
the order of~$\a$ modulo~$\gp$ for varying primes ideals~$\gp$, which
means the order of~$\a\bmod\gp$ as a torsion point in the underlying
group. The dynamical analog of the set of torsion points
is the set of preperiodic points, so
a natural dynamical analog of the classical Zsigmnody theorems
would be  to look at primitive prime divisors in the sequence
\[
  \f^n(\a) - \a,
  \qquad\text{or, more generally,}\quad
  \f^{m+n}(\a)-\f^m(\a).
\]
We formulate two primitive divisor conjectures, analogous to
Theorem~\ref{thm:dynBZ}, corresponding to these situations.
\end{remark}

\begin{conjecture}
\label{conj:weakconj}
Let~$K$ be a number field, let~$\f(z)\in
K(z)$ be a rational function of degree~$d\ge2$,
and let~$\a\in K$ be a $\f$-wandering point.
For each~$n\ge1$, write the ideal
\[
  \bigl(\f^n(\a)-\a\bigr) = \gA_n\gB_n^{-1}
\]
as a quotient of relatively prime integral ideals.  Then the dynamical
Zsigmondy set $\Zcal\bigl((\gA_n)_{n\ge1}\bigr)$ is finite.
\end{conjecture}

Note that as discussed in Remark~\ref{remark:gwandering}, there is no
problem with the growth of~$\Norm\gA_n$.  However, there is a
potential problem of primes reappering in the sequence to very high
powers. For example, if we take~$\f(z)=p-z+p^ez^2$ and~$\a=0$, then
\[
  \f(\a) = p\quad\text{and}\quad \f^2(\a)=\f(p)=p^{e+2}.
\]

In order to state the second conjecture, we need to extend the
definition of Zsigmondy sets to doubly indexed sequences.

\begin{definition}
Let 
\[
  (\gA_{m,n})_{\substack{m\ge1\\n\ge0\\}}
\]
be a doubly indexed sequence of ideals.  We say that~$\gp$ is a
\emph{primitive prime divisor} of~$\gA_{m,n}$ if
\[
  \gp\mid\gA_{m,n}
  \quad\text{and}\quad
  \gp\nmid\gA_{i,j}
  \quad\text{for all $i,j$ with $0\le i<m$ or $1\le j<n$.}
\]
The we define the \emph{Zsigmondy set~$\Zcal(\gA_{m,n})$} to be the set
\[
  \bigl\{(m,n):\text{$n\ge1$, $m\ge0$, and
            $\gA_{m,n}$ has no primitive divisors}\bigr\}.
\]
\end{definition}

\begin{conjecture}
\label{conj:strongconj}
Let~$K$ be a number field, let~$\f(z)\in K(z)$ be a rational function
of degree~$d\ge2$, and let~$\a\in K$ be a $\f$-wandering point.  For
each~$n\ge1$ and~$m\ge0$, write the ideal
\begin{equation}
  \label{eqn:fmnafma}
  \bigl(\f^{m+n}(\a)-\f^m(\a)\bigr) = \gA_{m,n}\gB_{m,n}^{-1}
\end{equation}
as a quotient of relatively prime integral ideals.  Then the dynamical
Zsigmondy set $\Zcal\bigl(\gA_{m,n}\bigr)$ is finite.
\end{conjecture}

\begin{remark}
In the classical multiplicative and elliptic Zsigmondy theorems, every
prime divides some term of the sequence, but this is (probably) not
true for (most) dynamically defined sequences.  In some nontrivial
dynamical cases, various authors~\cite{jones:denprimdiv,MR805714,MR813379}
have proven that
\[
  \bigl\{\gp:\gp\mid\gA_n~\text{for some $n$}\bigr\}
\]
is a set of density~$0$.  See also~\cite{MR917803} for a weak lower
bound on the number of primes in this set.
\par
On the other hand, it is clear that every prime divides at least one
term in the doubly indexed dynamical sequence~$\gA_{m,n}$ defined
by~\eqref{eqn:fmnafma}.  In particular, if~$\f$ has good reduction
at~$\gp$, then~$\gp\mid\gA_{m,n}$ if and only if the orbit of~$\a$
mdoulo~$\gp$ has a tail of length~$m$ and a cycle of length~$n$.
\end{remark}



\end{document}